\documentclass[runningheads]{llncs}

\usepackage[T1]{fontenc}
\usepackage{amsmath,amssymb,mathtools}
\usepackage{booktabs,array}
\usepackage{graphicx}
\usepackage[hidelinks]{hyperref}
\usepackage{tikz}
\usetikzlibrary{
  positioning,
  fit,
  calc,
  arrows.meta,
  backgrounds,
}

\newcommand{\F}{\mathbb F}

\newcommand{\pf}{\operatorname{pf}}
\newcommand{\Perf}{\operatorname{PerfMatch}}
\newcommand{\sgn}{\operatorname{sgn}}

\title{An Odd Pfaffian Number}
\titlerunning{An Odd Pfaffian Number}

\author{Priyanshu Pant\orcidID{0009-0004-1210-0938} \and Ranveer Singh\orcidID{0000-0002-4566-0331}}
\authorrunning{P. Pant and R. Singh}

\institute{
Indian Institute of Technology Indore, India\\
\email{priyanshupant03@gmail.com} (corresponding author),
\email{ranveer@iiti.ac.in}}

\begin{document}

\maketitle

\begin{abstract}
The Pfaffian number of a graph is the minimum number of Pfaffians needed to obtain its perfect-matching polynomial by linear combination.
In 2009, Norine conjectured that every Pfaffian number is a power of four.
Miranda and Lucchesi disproved this conjecture in 2011 by constructing a graph of Pfaffian number six, and conjectured instead that every nontrivial Pfaffian number is even.
We disprove their conjecture by proving that the Pfaffian number of \(K_{3,3}\sqcup K_{3,3}\) is \(13\).
We also construct a connected cubic bipartite matching-covered graph with Pfaffian number \(13\).

% We disprove their conjecture by proving that the Pfaffian number of \(K_{3,3}\sqcup K_{3,3}\) is \(13\).
% We also construct a connected cubic bipartite matching-covered graph with Pfaffian number \(13\).

\end{abstract}

\section{Introduction}
\label{sec:introduction}

All graphs considered in this paper are finite and simple. For a graph
\(G\), we write \(V(G)\) and \(E(G)\) for its vertex and edge sets. A
\emph{perfect matching} of \(G\) is a set of pairwise disjoint edges
covering every vertex, and we denote the set of perfect matchings by
\(\mathcal M(G)\). 
Counting perfect matchings is a fundamental problem in graph theory and
combinatorics~\cite{lovasz1986matching,lucchesi2024perfect}. For an \(n\times n\) matrix \(A=(a_{ij})\), the permanent
and determinant are
\[
\operatorname{per}(A)
=
\sum_{\sigma\in S_n}\prod_{i=1}^{n}a_{i,\sigma(i)},
\qquad
\det(A)
=
\sum_{\sigma\in S_n}
\sgn(\sigma)\prod_{i=1}^{n}a_{i,\sigma(i)}.
\]
For bipartite graphs, counting perfect matchings is equivalent to
computing the permanent of a \(0\)-\(1\) matrix. Although the permanent
and determinant differ only in the signs of their terms, determinants
can be computed in polynomial time, whereas computing the permanent of
a \(0\)-\(1\) matrix is \(\#P\)-complete
\cite{valiant1979complexity}.
Pfaffian orientations provide a graph-theoretic way of replacing the
permanent by a determinant. Let \(G\) be a graph on the
ordered vertex set
\(
V(G)=\{1,2,\ldots,2n\},
\)
and assign an indeterminate \(x_e\) to each edge \(e\in E(G)\). Its
\emph{perfect-matching polynomial} is
\[
\Perf(G)
=
\sum_{M\in\mathcal M(G)}
\prod_{e\in M}x_e.
\]
Let \(D\) be an \emph{orientation} of \(G\), that is, a choice of direction for
each edge. For a
perfect matching
\(
M=\{e_1,\ldots,e_n\},
\)
write \(e_i=u_iv_i\), where \(e_i\) is oriented from \(u_i\) to \(v_i\) in \(D\). Define
\[
\pi_D(M)
=
\begin{pmatrix}
1&2&3&4&\cdots&2n-1&2n\\
u_1&v_1&u_2&v_2&\cdots&u_n&v_n
\end{pmatrix}.
\]
The \emph{sign of \(M\) with respect to \(D\)}, denoted by
\(\sgn_D(M)\), is given by \(\sgn(\pi_D(M))\). This does not depend on the ordering of the edges of
\(M\), since interchanging two edge-pairs is an even permutation. For an edge \(e=ij\), we write \(x_e\) and \(x_{ij}\) interchangeably.
The \emph{weighted skew-adjacency matrix} associated with \(D\)
is the \(2n\times2n\) matrix \(A_D\) whose entries are
\[
(A_D)_{ij}
=
\begin{cases}
 x_{ij},
 &\text{if \(ij\in E(G)\) is oriented from \(i\) to \(j\)},\\
-x_{ij},
 &\text{if \(ij\in E(G)\) is oriented from \(j\) to \(i\)},\\
0,
 &\text{if \(ij\notin E(G)\)}.
\end{cases}
\]
The Pfaffian of \(A_D\) has the expansion
\[
\operatorname{Pf}(A_D)
=
\sum_{M\in\mathcal M(G)}
\sgn_D(M)\prod_{e\in M}x_e.
\]
Thus \(\operatorname{Pf}(A_D)\) contains the same monomials as
\(\Perf(G)\), but the coefficient of each perfect-matching monomial is
either \(+1\) or \(-1\), depending on the orientation \(D\).
If there is an orientation \(D\) for which all perfect matchings have the
same sign, then
\[
\Perf(G)=\pm\operatorname{Pf}(A_D).
\]
Such an orientation is called a \emph{Pfaffian orientation}, and a graph
admitting one is called \emph{Pfaffian}. Since
\[
\operatorname{Pf}(A_D)^2=\det(A_D),
\]
a Pfaffian orientation, when it exists, gives a polynomial-time method for counting the
perfect matchings of \(G\) \cite{cayley1849determinants}.

Kasteleyn proved that every planar graph is Pfaffian
\cite{kasteleyn1967graph}. Thus, every planar graph admits a Pfaffian
representation of its perfect-matching polynomial. For bipartite graphs, recognizing
Pfaffian graphs can be done in polynomial time by the work of
Robertson, Seymour and Thomas~\cite{robertson1999permanents} and
independently McCuaig~\cite{mccuaig2004polya}.
For general graphs, a randomized polynomial-time recognition algorithm was recently given by Pant and Singh~\cite{PantSinghPfaffianRecognition}. Whether a deterministic polynomial-time algorithm exists remains open.

Not every graph is Pfaffian, in particular, \(K_{3,3}\) is
non-Pfaffian, meaning that no orientation makes all of its perfect
matchings have the same sign~\cite{little1975characterization,norine2009drawing}.
For a non-Pfaffian graph, one may ask whether its perfect-matching
polynomial can still be obtained as a linear combination of several
Pfaffians. Following Norine~\cite{norine2009drawing}, a graph \(G\) is
\emph{\(k\)-Pfaffian} if there exist orientations
\(D_1,\ldots,D_k\) and real coefficients \(c_1,\ldots,c_k\) such that
\[
\Perf(G)
=
\sum_{i=1}^{k} c_i\operatorname{Pf}(A_{D_i}).
\]
The \emph{Pfaffian number} of \(G\), denoted by \(\pf(G)\), is the
smallest such \(k\).
Thus \(\pf(G)=1\) if and only if \(G\) is
Pfaffian. 
Computing the exact Pfaffian number \(\pf(G)\) generalizes Pfaffian recognition, which asks only whether \(\pf(G)=1\).
To prove an upper bound on \(\pf(G)\), it suffices to give one such
representation, whereas proving a lower bound is difficult,
as one must rule out every representation using fewer Pfaffians.
Even for small graphs, the exact Pfaffian number can therefore be hard
to determine.

Norine proved that every \(3\)-Pfaffian graph is Pfaffian and every
\(5\)-Pfaffian graph is \(4\)-Pfaffian~\cite{norine2009drawing}. 
In particular,
\(
\pf(K_{3,3})=4.
\)
Motivated by these results, Norine conjectured that every Pfaffian number
is a power of four. Miranda
and Lucchesi disproved this conjecture by constructing a graph of
Pfaffian number \(6\), and conjectured instead that every Pfaffian
number greater than one is even~\cite{miranda2011matching}.
Subsequently, Costa Moço, Miranda, and Nunes da Silva characterized the signature matrices associated with Pfaffian number \(6\)~\cite{moco2021signature}.

More recently, general lower bounds on the Pfaffian number have been obtained. Junchaya, Miranda, and Lucchesi showed, in particular, that the Pfaffian number is unbounded~\cite{junchaya2026lower}, while exponential lower bounds, including for connected matching-covered graphs, were established in~\cite{pant2026exponential}. 
These results concern how large the Pfaffian number can be, but leave
open the basic question of which exact values can occur.
Since \(\pf(K_{3,3})=4\), taking the product of two four-term
representations gives
$
\pf(K_{3,3}\sqcup K_{3,3})\le16,
$
where \(K_{3,3}\sqcup K_{3,3}\) denotes the disjoint union of two copies
of \(K_{3,3}\).
It is therefore natural to ask whether the sixteen-term product representation is optimal.  Our main result gives the exact answer.

\begin{theorem}
\label{thm:main}
\(\pf(K_{3,3}\sqcup K_{3,3})=13.\)
\end{theorem}

Miranda and Nunes da Silva independently obtained the value \(13\) using integer linear programming (Miranda, personal communication, August 2026). Thus the product upper bound of \(16\) is not optimal. More importantly, since \(13\) is odd, Theorem~\ref{thm:main} disproves the conjecture of Miranda and Lucchesi~\cite{miranda2011matching}. Their further conjecture, that all coefficient magnitudes in minimum representations of graphs with even Pfaffian number are \(1/2\), remains open. Our minimum thirteen-term representation has one coefficient of magnitude \(1/2\) and twelve of magnitude \(1/4\).

We further show that the Pfaffian number \(13\) is also attained by a
connected cubic bipartite matching-covered graph. Recall that a graph is
\emph{cubic} if every vertex has degree \(3\), and a graph is
\emph{matching-covered} if it is connected and every edge belongs to a
perfect matching. A \emph{bisubdivision} of a graph is obtained by subdividing some edges
an even number of times.

Let \(\widetilde G\) be the graph obtained from two copies of \(K_{3,3}\) by
bisubdividing the edges \(a_i\alpha_i\) and adding the crossed edges
\(x_1y_2\) and \(x_2y_1\) as in Figure~\ref{fig:Gtilde}.

\begin{figure}[!ht]
\centering
\begin{tikzpicture}[
scale=0.5,
vertex/.style={circle, draw, fill=white, inner sep=1.3pt},
every node/.style={font=\small}
]

% Left copy
\node[vertex] (c1) at (-5.2,1.4) {};
\node[vertex] (b1) at (-4.0,1.4) {};
\node[vertex] (a1) at (-2.8,1.4) {};

\node[vertex] (C1) at (-5.2,-1.4) {};
\node[vertex] (B1) at (-4.0,-1.4) {};
\node[vertex] (A1) at (-2.8,-1.4) {};

\node[above] at (c1) {$c_1$};
\node[above] at (b1) {$b_1$};
\node[above] at (a1) {$a_1$};

\node[below] at (C1) {$\gamma_1$};
\node[below] at (B1) {$\beta_1$};
\node[below] at (A1) {$\alpha_1$};

\draw (c1)--(C1);
\draw (c1)--(B1);
\draw (c1)--(A1);
\draw (b1)--(C1);
\draw (b1)--(B1);
\draw (b1)--(A1);
\draw (a1)--(C1);
\draw (a1)--(B1);

\node[vertex] (x1) at (-1.5,0.9) {};
\node[vertex] (y1) at (-1.5,-0.9) {};

\node[above right] at (x1) {$x_1$};
\node[below right] at (y1) {$y_1$};

\draw (a1)--(x1)--(y1)--(A1);

% Right copy
\node[vertex] (a2) at (2.8,1.4) {};
\node[vertex] (b2) at (4.0,1.4) {};
\node[vertex] (c2) at (5.2,1.4) {};

\node[vertex] (A2) at (2.8,-1.4) {};
\node[vertex] (B2) at (4.0,-1.4) {};
\node[vertex] (C2) at (5.2,-1.4) {};

\node[above] at (a2) {$a_2$};
\node[above] at (b2) {$b_2$};
\node[above] at (c2) {$c_2$};

\node[below] at (A2) {$\alpha_2$};
\node[below] at (B2) {$\beta_2$};
\node[below] at (C2) {$\gamma_2$};

\draw (a2)--(B2);
\draw (a2)--(C2);
\draw (b2)--(A2);
\draw (b2)--(B2);
\draw (b2)--(C2);
\draw (c2)--(A2);
\draw (c2)--(B2);
\draw (c2)--(C2);

\node[vertex] (x2) at (1.5,0.9) {};
\node[vertex] (y2) at (1.5,-0.9) {};

\node[above left] at (x2) {$x_2$};
\node[below left] at (y2) {$y_2$};

\draw (a2)--(x2)--(y2)--(A2);

% Crossed edges
\draw[thick] (x1)--(y2);
\draw[thick] (y1)--(x2);

\end{tikzpicture}
\caption{The graph \(\widetilde G\).}
\label{fig:Gtilde}
\end{figure}

A subgraph \(H\) of a graph \(G\) is called a \emph{spanning subgraph} if \(V(H)=V(G)\). A subgraph \(H\) of \(G\) is called \emph{conformal} if both \(H\) and \(G-V(H)\) have perfect matchings, where the empty graph is regarded as having a perfect matching.
We use two facts proved in \cite{pant2026exponential}. First, if \(H\) is a conformal subgraph of \(G\), then
$
\pf(G)\ge\pf(H).
$
Second, bisubdivision does not change the Pfaffian number of a disjoint
union of matching-covered graphs.

\begin{theorem}
\label{thm:connected-matching-covered}
\(\widetilde G\) is cubic, bipartite, and matching-covered, and
\(\pf(\widetilde G)=13\).
\end{theorem}

\begin{proof}
It is immediate from the construction that \(\widetilde G\) is
connected, cubic, and bipartite. It is also matching-covered since every
edge in either bisubdivided copy lies in a perfect matching avoiding the
crossed edges, while the two crossed edges lie together in a perfect
matching.
For the lower bound, delete the two crossed edges \(x_1y_2\) and
\(x_2y_1\) from \(\widetilde G\), and let \(H\) be the resulting
subgraph. Then \(H\) is a bisubdivision of
\(K_{3,3}\sqcup K_{3,3}\). Since \(H\) is spanning and has a perfect
matching, it is conformal in \(\widetilde G\). Therefore
\[
\pf(\widetilde G)
\ge
\pf(H)
=
\pf(K_{3,3}\sqcup K_{3,3})
=
13,
\]
where the last equality follows from Theorem~\ref{thm:main}.
The upper bound \(\pf(\widetilde G)\le 13\) is proved in
Appendix~\ref{app:connected-upper}, completing the proof.
\end{proof}

The remainder of the paper is organized as follows. Section~\ref{sec:preliminaries}
develops the matrix formulation used throughout the paper.
Section~\ref{sec:upper-bound} proves the upper bound, and
Section~\ref{sec:lower-bound} proves the matching lower bound.
The omitted proofs are collected in the appendix.

\section{Preliminaries}
\label{sec:preliminaries}

We first translate the Pfaffian-number problem into a matrix problem. For this, we use signed determinants and the special structure of \(K_{3,3}\).

\subsection{Bipartite graphs and signed determinants}

Let \(G=(U\sqcup V,E)\) be a bipartite graph with \(|U|=|V|=n\).
Its \emph{weighted biadjacency matrix}, denoted by \(B_G\), is defined by
\[
(B_G)_{uv}
=
\begin{cases}
x_{uv},&\text{if }uv\in E(G),\\
0,&\text{otherwise}.
\end{cases}
\]
We write \(\circ\) for the entrywise product of matrices of the same
size. Then
\[
\Perf(G)=\operatorname{per}(B_G).
\]
For bipartite graphs, Pfaffian representations can be written in terms
of signed determinants; see
\cite[Lemmas~2.1 and~2.2]{pant2026exponential}.

\begin{lemma}
\label{lem:bipartite-matrix-formulation}
The Pfaffian number \(\pf(G)\) is the least \(k\ge1\) for which there exist \(S_i\in\{\pm1\}^{n\times n}\) and \(c_i\in\mathbb R\) satisfying
$$
\operatorname{per}(B_G)=\sum_{i=1}^{k}c_i\det(S_i\circ B_G).
$$
\end{lemma}

\subsection{The coefficient vectors of \(K_{3,3}\)}

Let \(B=(x_{ij})_{i,j\in[3]}\) be the weighted biadjacency matrix of
\(K_{3,3}\). We order the six permutations in \(S_3\) as
\begin{equation}
123,\quad132,\quad213,\quad231,\quad312,\quad321.
\label{eq:permutation-order}
\end{equation}
For \(\pi\in S_3\), write
\(x_\pi=\prod_{i=1}^{3}x_{i,\pi(i)}\). Then
\(
\operatorname{per}(B)=\sum_{\pi\in S_3}x_\pi.
\)
For a sign matrix \(S=(s_{ij})\in\{\pm1\}^{3\times3}\), define
\(v_S\in\{\pm1\}^{6}\) by
\[
v_S(\pi)
=
\sgn(\pi)\prod_{i=1}^{3}s_{i,\pi(i)}.
\]
Thus
\(
\det(S\circ B)
=
\sum_{\pi\in S_3}v_S(\pi)x_\pi.
\)
The next lemma characterizes exactly which sign vectors can occur.

\begin{lemma}
\label{lem:one-copy-coefficient-vectors}
A vector \(v\in\{\pm1\}^{6}\) is equal to \(v_S\) for some
\(S\in\{\pm1\}^{3\times3}\) if and only if
\(\prod_{j=1}^{6}v_j=-1\).
\end{lemma}

The proof is given in Appendix~\ref{app:coefficient-vectors}.
By Lemma~\ref{lem:one-copy-coefficient-vectors}, the coefficient vectors
of signed \(3\times3\) determinants are precisely the vectors in
\(\{\pm1\}^6\) whose coordinate product is \(-1\). Since an overall sign
can be included in the scalar coefficient, we normalize these vectors by
requiring their first coordinate to be \(1\).

For \(a=(a_1,a_2,a_3,a_4)\in \F_2^4\), define
\[
o_a
=
\begin{pmatrix}
1, &
(-1)^{a_1}, &
(-1)^{a_2}, &
(-1)^{a_3}, &
(-1)^{a_4}, &
-(-1)^{a_1+a_2+a_3+a_4}
\end{pmatrix}^{\mathsf T}.
\label{eq:odd-vector}
\]

\begin{corollary}
\label{cor:normalized-coefficient-vectors}
For every \(S\in\{\pm1\}^{3\times3}\), there exist unique
\(\varepsilon\in\{\pm1\}\) and \(a\in \F_2^4\) such that
\[
v_S=\varepsilon\,o_a.
\]
\end{corollary}
The proof is also given in Appendix~\ref{app:coefficient-vectors}.

\subsection{The matrix formulation}

Let \(B_1=(x_{ij})\) and \(B_2=(y_{ij})\) be the weighted biadjacency
matrices of two copies of \(K_{3,3}\). For \(\pi,\tau\in S_3\), write
\(x_\pi=\prod_{i=1}^{3}x_{i,\pi(i)}\) and
\(y_\tau=\prod_{i=1}^{3}y_{i,\tau(i)}\).
For a polynomial
\(F=\sum_{\pi,\tau\in S_3}q_{\pi,\tau}x_\pi y_\tau\), define its
\emph{coefficient matrix} by
\(C(F)=(q_{\pi,\tau})_{\pi,\tau\in S_3}\), with rows and columns indexed
by \(S_3\) in the order fixed above. Thus the \((\pi,\tau)\)-entry of
\(C(F)\) is the coefficient of \(x_\pi y_\tau\) in \(F\).
Let \(J_6\) denote the \(6\times6\) all-ones matrix.

\begin{lemma}
\label{lem:finite-formulation}
The Pfaffian number \(\pf(K_{3,3}\sqcup K_{3,3})\) is the least
\(k\ge1\) for which there exist \(a_t,b_t\in \F_2^4\) and \(c_t\in\mathbb R\)
such that

$$
J_6=\sum_{t=1}^{k}c_t\,o_{a_t}o_{b_t}^{\mathsf T}.
$$

\end{lemma}

\begin{proof}
The weighted biadjacency matrix of \(K_{3,3}\sqcup K_{3,3}\) is
\(B_1\oplus B_2\), and hence
$$
\operatorname{per}(B_1\oplus B_2)
=\operatorname{per}(B_1)\operatorname{per}(B_2)
=\sum_{\pi,\tau\in S_3}x_\pi y_\tau.
$$
Therefore
\(C(\operatorname{per}(B_1)\operatorname{per}(B_2))=J_6\).
Similarly, every signed determinant of \(B_1\oplus B_2\) factors as
\(\det(S\circ B_1)\det(T\circ B_2)\) for some
\(S,T\in\{\pm1\}^{3\times3}\). The coefficient of \(x_\pi y_\tau\) is
\(v_S(\pi)v_T(\tau)\), so its coefficient matrix is
\(v_Sv_T^{\mathsf T}\).
By Corollary~\ref{cor:normalized-coefficient-vectors},
\(v_S=\varepsilon o_a\) and \(v_T=\delta o_b\) for some
\(a,b\in \F_2^4\) and \(\varepsilon,\delta\in\{\pm1\}\). Thus \(v_Sv_T^{\mathsf T}=\varepsilon\delta\,o_a o_b^{\mathsf T}\).
The sign \(\varepsilon\delta\) can be included in the scalar coefficient,
and the result follows from Lemma~\ref{lem:bipartite-matrix-formulation}.
\end{proof}

\section{The upper bound}
\label{sec:upper-bound}

Taking the product of two four-term representations gives sixteen terms. To improve this bound, we choose the pairs \((a,b)\) together, rather than separately for each copy of \(K_{3,3}\). By Lemma~\ref{lem:finite-formulation}, the following thirteen terms give the required representation of \(J_6\). For a bit string \(a_1a_2a_3a_4\), we write
\(o_{a_1a_2a_3a_4}\) for \(o_{(a_1,a_2,a_3,a_4)}\).
The thirteen triples \((\mu_t,a_t,b_t)\) are listed in Table~\ref{tab:certificate}.

\begin{table}[ht]
\centering
\caption{The thirteen triples \((\mu_t,a_t,b_t)\).}
\label{tab:certificate}
\small
\setlength{\tabcolsep}{4pt}
\begin{tabular}{crcc@{\qquad}crcc}
\toprule
\(t\) & \(\mu_t\) & \(a_t\) & \(b_t\) &
\(t\) & \(\mu_t\) & \(a_t\) & \(b_t\)\\
\midrule
1 &  1 & 1111 & 1111 & 8  & -1 & 1001 & 0111\\
2 &  1 & 0000 & 0100 & 9  & -1 & 0101 & 0101\\
3 &  1 & 0100 & 0000 & 10 & -1 & 1101 & 0011\\
4 &  2 & 0001 & 0001 & 11 &  1 & 1100 & 0001\\
5 & -1 & 1111 & 1000 & 12 &  1 & 0010 & 1000\\
6 &  1 & 1000 & 0010 & 13 &  1 & 0001 & 0110\\
7 & -1 & 0010 & 1111 &    &     &      &     \\
\bottomrule
\end{tabular}
\end{table}

\begin{theorem}[Upper bound]
\label{thm:upper-bound}
\(
\pf(K_{3,3}\sqcup K_{3,3})\le13.
\)
\end{theorem}

\begin{proof}
For \(1\le t\le13\), let \(P,Q\in\mathbb R^{13\times6}\) be defined row-wise by
\(P_{t,*}=\mu_t o_{a_t}^{\mathsf T}\) and
\(Q_{t,*}=o_{b_t}^{\mathsf T}\), where
\((\mu_t,a_t,b_t)\) are the triples in Table~\ref{tab:certificate}. Then \(P^{\mathsf T}Q=\sum_{t=1}^{13}\mu_t\,o_{a_t}o_{b_t}^{\mathsf T}\).
Using the definition of \(o_a\), we obtain
{\scriptsize
\[
P=
\begin{pmatrix}
 1&-1&-1&-1&-1&-1\\
 1& 1& 1& 1& 1&-1\\
 1& 1&-1& 1& 1& 1\\
 2& 2& 2& 2&-2& 2\\
-1& 1& 1& 1& 1& 1\\
 1&-1& 1& 1& 1& 1\\
-1&-1&-1& 1&-1&-1\\
-1& 1&-1&-1& 1& 1\\
-1&-1& 1&-1& 1& 1\\
-1& 1& 1&-1& 1&-1\\
 1&-1&-1& 1& 1&-1\\
 1& 1& 1&-1& 1& 1\\
 1& 1& 1& 1&-1& 1
\end{pmatrix},
\qquad
Q=
\begin{pmatrix}
 1&-1&-1&-1&-1&-1\\
 1& 1&-1& 1& 1& 1\\
 1& 1& 1& 1& 1&-1\\
 1& 1& 1& 1&-1& 1\\
 1&-1& 1& 1& 1& 1\\
 1& 1& 1&-1& 1& 1\\
 1&-1&-1&-1&-1&-1\\
 1& 1&-1&-1&-1& 1\\
 1& 1&-1& 1&-1&-1\\
 1& 1& 1&-1&-1&-1\\
 1& 1& 1& 1&-1& 1\\
 1&-1& 1& 1& 1& 1\\
 1& 1&-1&-1& 1&-1
\end{pmatrix}.
\]
}
Direct multiplication gives
\(
P^{\mathsf T}Q
=
4J_6.
\)
Hence \(4J_6=\sum_{t=1}^{13}\mu_t\,o_{a_t}o_{b_t}^{\mathsf T}\).
Dividing by \(4\) and applying
Lemma~\ref{lem:finite-formulation} gives
\(
\pf(K_{3,3}\sqcup K_{3,3})\le13.
\)
\end{proof}

\section{The lower bound}
\label{sec:lower-bound}

For the lower bound, we express the matrix identity in terms of functions on \(\F_2^4\). Their supports form a collection of blocks with simple intersection properties. We first use these properties to restrict a representation with at most twelve terms, and then use the function values to obtain a contradiction.

We now prove that
\(
\pf(K_{3,3}\sqcup K_{3,3})\ge13.
\)
Suppose, for contradiction, that
\(
\pf(K_{3,3}\sqcup K_{3,3})\le12.
\)
By Lemma~\ref{lem:finite-formulation}, there exist an integer
\(k\le12\), elements \(a_t,b_t\in \F_2^4\), and nonzero coefficients
\(c_t\in\mathbb R\), for \(1\le t\le k\), such that
\begin{equation}
J_6
=
\sum_{t=1}^{k}
c_t\,o_{a_t}o_{b_t}^{\mathsf T}.
\label{eq:hypothetical-identity}
\end{equation}
By combining terms having the same ordered pair \((a_t,b_t)\), and
deleting any resulting zero coefficients, we may assume that the ordered
pairs
\(
(a_1,b_1),\allowbreak\ldots,\allowbreak(a_k,b_k)
\)
are distinct.

\subsection{From matrices to functions}

Our first step is to convert the matrix identity
\eqref{eq:hypothetical-identity} into a family of scalar identities
indexed by \(\F_2^4\). For
\(x=(x_1,x_2,x_3,x_4)\in \F_2^4\), define
\[
e_x
=
\begin{pmatrix}
1 ,&
(-1)^{x_1} ,&
(-1)^{x_2}, &
(-1)^{x_3} ,&
(-1)^{x_4} ,&
(-1)^{x_1+x_2+x_3+x_4}
\end{pmatrix}^{\mathsf T}.
\]
In particular,
\(
e_{\mathbf0}
=
(1,1,1,1,1,1)^{\mathsf T},
\)
and hence
\(
J_6=e_{\mathbf0}e_{\mathbf0}^{\mathsf T}.
\)
For \(x\in \F_2^4\), let \(|x|\) denote its Hamming weight, that is, the
number of coordinates of \(x\) equal to \(1\). A direct
calculation gives
\begin{equation}
e_x^{\mathsf T}e_{\mathbf0}
=
5-2|x|+(-1)^{|x|}
=
\begin{cases}
6,&|x|=0,\\
2,&|x|=1\text{ or }2,\\
-2,&|x|=3\text{ or }4.
\end{cases}
\label{eq:even-even-inner-product}
\end{equation}
In particular, \(e_x^{\mathsf T}e_{\mathbf0}\neq0\) for every \(x\in \F_2^4\).
We now normalize the inner products \(e_x^{\mathsf T}o_a\). The
normalization is chosen precisely so that
\eqref{eq:hypothetical-identity} becomes an identity with constant
left-hand side. For \(a,x\in \F_2^4\), define
\begin{equation}
f_a(x)
=
\frac{e_x^{\mathsf T}o_a}
     {2e_x^{\mathsf T}e_{\mathbf0}}.
\label{eq:normalized-function}
\end{equation}

\begin{lemma}
\label{lem:functional-reformulation}
For every \(x,y\in \F_2^4\), with \(\lambda_t=4c_t\neq0\),
\begin{equation}
\label{eq:normalized-identity}
1=\sum_{t=1}^{k}\lambda_t f_{a_t}(x)f_{b_t}(y).
\end{equation}
\end{lemma}

\begin{proof}
Multiply \eqref{eq:hypothetical-identity} on the left by
\(e_x^{\mathsf T}\) and on the right by \(e_y\). Since
\(J_6=e_{\mathbf0}e_{\mathbf0}^{\mathsf T}\), the left-hand side is
\[
e_x^{\mathsf T}J_6e_y
=
e_x^{\mathsf T}
\bigl(e_{\mathbf0}e_{\mathbf0}^{\mathsf T}\bigr)e_y
=
\bigl(e_x^{\mathsf T}e_{\mathbf0}\bigr)
\bigl(e_y^{\mathsf T}e_{\mathbf0}\bigr).
\]
On the other hand, by \eqref{eq:normalized-function},
\[
e_x^{\mathsf T}o_{a_t}
=
2(e_x^{\mathsf T}e_{\mathbf0})f_{a_t}(x)
\text{ and }
e_y^{\mathsf T}o_{b_t}
=
2(e_y^{\mathsf T}e_{\mathbf0})f_{b_t}(y).
\]
Therefore
\[
(e_x^{\mathsf T}e_{\mathbf0})
(e_y^{\mathsf T}e_{\mathbf0})
=
4(e_x^{\mathsf T}e_{\mathbf0})
(e_y^{\mathsf T}e_{\mathbf0})
\sum_{t=1}^{k}
c_t f_{a_t}(x)f_{b_t}(y).
\]
By \eqref{eq:even-even-inner-product}, both common factors are nonzero.
Dividing by them gives
\[
1
=
\sum_{t=1}^{k}
4c_t f_{a_t}(x)f_{b_t}(y).
\]
\end{proof}

\subsection{The block system}

Lemma~\ref{lem:functional-reformulation} replaces the original matrix
decomposition by functions on the sixteen-point space \(\F_2^4\). We next
determine exactly where these functions are nonzero and what values they
take there.
Write \(\mathbf0=(0,0,0,0)\) and \(\mathbf1=(1,1,1,1)\)
and let
\(\varepsilon_1,\varepsilon_2,\varepsilon_3,\varepsilon_4\) be the standard basis vectors of
\(\F_2^4\). Define
\(
D
=
\{\mathbf0,\varepsilon_1,\varepsilon_2,
\varepsilon_3,\varepsilon_4,\mathbf1\},
\)
and, for \(a\in \F_2^4\), let
\(
B_a=a+D.
\)
We call each set \(B_a\) a \emph{block}. For a function \(f:\F_2^4\to\mathbb R\), write
\[
\operatorname{supp}(f)
=
\{x\in \F_2^4:f(x)\neq0\}.
\] 
The next lemma describes the sets \(B_a\) arising as supports of the functions \(f_a\).

\begin{lemma}
\label{lem:block-system}
The following statements hold.
\begin{enumerate}
\item[(i)] For every \(a\in \F_2^4\), \(\operatorname{supp}(f_a)=B_a\).
\item[(ii)] For every \(a,x\in \F_2^4\) with \(x\in B_a\),
\(
|f_a(x)|
=
\begin{cases}
\dfrac13,&x=\mathbf0,\\
1,&x\neq\mathbf0.
\end{cases}
\)
\item[(iii)] Every block has size six, that is, \(|B_a|=6\).
\item[(iv)] Any two distinct blocks meet in exactly two points,
\(|B_a\cap B_b|=2\) for \(a\neq b\).
\item[(v)] Every point belongs to exactly six blocks,
\(\bigl|\{a\in \F_2^4:x\in B_a\}\bigr|=6\) for every \(x\in \F_2^4\).
\end{enumerate}
\end{lemma}
The proof of Lemma~\ref{lem:block-system} is given in Appendix~\ref{app:block-system}.

\subsection{The rectangle cover}
\label{subsec:rectangle-cover}

We now use the supports of the functions \(f_a\). For \(a,b\in \F_2^4\),
we call a set of the form \(B_a\times B_b\subseteq \F_2^4\times \F_2^4\) a
\emph{rectangle}. By Lemma~\ref{lem:functional-reformulation} and
Lemma~\ref{lem:block-system}, the selected rectangles cover
\(\F_2^4\times \F_2^4\).

\begin{lemma}
\label{lem:rectangle-cover}
The rectangles \(B_{a_t}\times B_{b_t}\), \(t\in[k]\), cover \(\F_2^4\times \F_2^4\), that is,
\[
\F_2^4\times \F_2^4
=
\bigcup_{t=1}^{k}
(B_{a_t}\times B_{b_t}).
\]
\end{lemma}
\begin{proof}
Fix \((x,y)\in \F_2^4\times \F_2^4\). By
Lemma~\ref{lem:functional-reformulation},
\[
1
=
\sum_{t=1}^{k}
\lambda_t f_{a_t}(x)f_{b_t}(y).
\]
Since the left-hand side is nonzero, at least one term on the
right-hand side is nonzero. By
Lemma~\ref{lem:block-system}\textup{(i)}, this means that, for some
\(t\in[k]\),
\(
x\in B_{a_t}
\text{ and }
y\in B_{b_t}.
\) 
Hence
\(
(x,y)\in B_{a_t}\times B_{b_t}.
\)
Since \((x,y)\) was arbitrary, the rectangles cover \(\F_2^4\times \F_2^4\).
\end{proof}

For \(x,y\in \F_2^4\), define the left and right degrees
\[
d_L(x)
=
\bigl|\{t\in[k]:x\in B_{a_t}\}\bigr|,
\qquad
d_R(y)
=
\bigl|\{t\in[k]:y\in B_{b_t}\}\bigr|.
\]
Thus \(d_L(x)\) is the number of selected left blocks containing \(x\),
while \(d_R(y)\) is the number of selected right blocks containing
\(y\).

\subsection{Forcing twelve terms}

We now use the rectangle cover to show that any representation with at
most twelve terms must in fact use exactly twelve terms. Moreover,
there is an unused block \(A\) such that every point of \(A\) has left
degree four.

\begin{lemma}
\label{lem:tight-unused-block}
Under the assumption \(k\le12\), the following statements hold.
\begin{enumerate}
\item[(i)] Every point has left and right degree at least four, that is, 
\(d_L(x)\ge4\) and \(d_R(x)\ge4\) for every \(x\in \F_2^4\).

\item[(ii)] We have \(k=12\).

\item[(iii)] There exists a block \(A\) that does not occur among
\(B_{a_1},\ldots,B_{a_{12}}\) and satisfies \(d_L(z)=4\) for every
\(z\in A\).
\end{enumerate}
\end{lemma}

\begin{proof}
We first note that three blocks cannot cover \(\F_2^4\). If
\(B_a,B_b,B_c\) are distinct, then by Lemma~\ref{lem:block-system},
\[
|B_a\cup B_b\cup B_c|
=
18-2-2-2+|B_a\cap B_b\cap B_c|
\le14.
\]
If some of the blocks coincide, their union is even smaller.

\emph{\textup{(i)}.}
Fix \(x\in \F_2^4\). By the rectangle cover, the blocks \(B_{b_t}\) with
\(x\in B_{a_t}\) cover \(\F_2^4\). Since three blocks cannot cover \(\F_2^4\),
at least four such blocks are needed. Hence \(d_L(x)\ge4\).
Interchanging the left and right sides gives \(d_R(x)\ge4\) for every
\(x\in \F_2^4\).

\smallskip
\emph{\textup{(ii)}.}
We first count incidences between the sixteen points of \(\F_2^4\) and the
selected left blocks \(B_{a_1},\ldots,B_{a_k}\). Each selected left
block contains six points, while a point \(x\in \F_2^4\) is contained in
exactly \(d_L(x)\) selected left blocks. Therefore
\begin{equation}
\sum_{x\in \F_2^4} d_L(x)
=
\sum_{t=1}^{k}|B_{a_t}|
=
6k
\le72.
\label{eq:left-degree-sum}
\end{equation}
By part \textup{(i)}, every point has left degree at least four. If
every point had left degree at least five, then
\(\sum_{x\in \F_2^4}d_L(x)\ge16\cdot5=80\),
contradicting \eqref{eq:left-degree-sum}. Hence there exists a point
\(x\in \F_2^4\) with \(d_L(x)=4\).
Now consider the six blocks containing \(x\). By
Lemma~\ref{lem:block-system}\textup{(v)}, there are exactly six such
block types. Since \(d_L(x)=4\), only four selected left block
occurrences contain \(x\). Hence at least one of the six blocks
containing \(x\) does not occur among the selected left blocks
\(B_{a_1},\ldots,B_{a_k}\). Fix such a block and call it \(A\).
Thus \(A\neq B_{a_t}\) for every \(t\in[k]\). By Lemma~\ref{lem:block-system}\textup{(iv)}, every selected left block
\(B_{a_t}\) meets \(A\) in exactly two points, so
\(|A\cap B_{a_t}|=2\) for every \(t\in[k]\).

We now count incidences between the six points of \(A\) and the selected
left blocks \(B_{a_1},\ldots,B_{a_k}\). Counting by points gives
\(\sum_{z\in A}d_L(z)\), while each selected left block contributes
exactly two incidences. Hence
\begin{equation}
\sum_{z\in A}d_L(z)
=
\sum_{t=1}^{k}|A\cap B_{a_t}|
=
2k.
\label{eq:unused-block-count}
\end{equation}
Every point of \(A\) has left degree at least four, and \(A\) contains
six points. Therefore
\(2k=\sum_{z\in A}d_L(z)\ge6\cdot4=24\).
Thus \(k\ge12\). Since our assumption is \(k\le12\), we
conclude that \(k=12\).

\smallskip
\emph{\textup{(iii)}.}
Now \eqref{eq:unused-block-count} gives
\(\sum_{z\in A}d_L(z)=24\). Since the six terms are all at least four,
each must equal four. Thus \(d_L(z)=4\) for every \(z\in A\), and by
construction \(A\) does not occur among the selected left blocks.
\end{proof}

\subsection{The degree-four case}
\label{subsec:degree-four-case}

The previous lemma gives an unused block \(A\) whose points all have
left degree four. We now examine what happens at any point \(z\) with
\(d_L(z)=4\). The four selected left blocks containing \(z\) correspond
to four right blocks that cover \(\F_2^4\). Using the actual values of the
functions \(f_a\), we show that exactly three of these right blocks
contain \(\mathbf0\).

\begin{lemma}
\label{lem:three-out-of-four}
If \(d_L(z)=4\), then among the four indices \(t\in[k]\) satisfying
\(z\in B_{a_t}\), exactly three satisfy \(\mathbf0\in B_{b_t}\).
\end{lemma}

\begin{proof}
Let
\(
T_z=\{t\in[k]:z\in B_{a_t}\}.
\) 
Then \(|T_z|=4\). By the rectangle cover, the corresponding right blocks
cover \(\F_2^4\), so
\[
\F_2^4=\bigcup_{t\in T_z}B_{b_t}.
\]
These four blocks are distinct, since otherwise at most three blocks
would cover \(\F_2^4\), which is impossible by the argument used in the
proof of Lemma~\ref{lem:tight-unused-block}.
Fix \(t\in T_z\). The other three blocks cover at most fourteen points,
whereas all four together cover all sixteen points of \(\F_2^4\). Hence
\(B_{b_t}\) contains at least two points belonging to none of the other
three blocks. We call such points \emph{private points} of \(B_{b_t}\).

Now set \(x=z\) in \eqref{eq:normalized-identity}. Since
\(f_{a_t}(z)=0\) for \(t\notin T_z\), we obtain
\begin{equation}
1=\sum_{t\in T_z}\alpha_t f_{b_t}(y),
\label{eq:local-identity}
\end{equation}
where \(\alpha_t=\lambda_t f_{a_t}(z)\neq0\).
For each \(t\in T_z\), choose a private point
\(v_t\in B_{b_t}\) with \(v_t\neq\mathbf0\). This is possible because
each \(B_{b_t}\) has at least two private points. Evaluating
\eqref{eq:local-identity} at \(y=v_t\), only the \(t\)-th term survives,
so
\[
1=\alpha_t f_{b_t}(v_t).
\]
Since \(v_t\neq\mathbf0\), Lemma~\ref{lem:block-system}\textup{(ii)}
gives \(|f_{b_t}(v_t)|=1\). Hence
\begin{equation}
|\alpha_t|=1
\text{ for every }t\in T_z.
\label{eq:alpha-unit}
\end{equation}
Let
\(
r=\bigl|\{t\in T_z:\mathbf0\in B_{b_t}\}\bigr|.
\)
Evaluate \eqref{eq:local-identity} at \(y=\mathbf0\). Exactly \(r\)
terms are nonzero, and by Lemma~\ref{lem:block-system}\textup{(ii)} and
\eqref{eq:alpha-unit}, each surviving term is either \(1/3\) or
\(-1/3\). Therefore
\[
1=
\underbrace{\pm\frac13\pm\cdots\pm\frac13}_{r\text{ terms}}
\;\Longrightarrow\;
3=
\underbrace{\pm1\pm\cdots\pm1}_{r\text{ terms}}.
\]
If \(r\le2\), the right-hand side has absolute value at most \(2\), while
if \(r=4\), it is even. Hence \(r=3\).
Thus exactly three of the four corresponding right blocks contain
\(\mathbf0\).
\end{proof}

\subsection{The final contradiction}
\label{subsec:handshake-contradiction}

We now combine the structure forced by a twelve-term representation with
the three-out-of-four rule to obtain the final contradiction.

\begin{theorem}[Lower bound]
\label{thm:lower-bound}
\(
\pf(K_{3,3}\sqcup K_{3,3})\ge13.
\)
\end{theorem}

\begin{proof}
Under the contradiction assumption, Lemma~\ref{lem:tight-unused-block}
gives \(k=12\) and an unused block \(A\) such that \(d_L(z)=4\) for
every \(z\in A\).
Let \(H=\{t\in[12]:\mathbf0\in B_{b_t}\}\).
Thus \(|H|=d_R(\mathbf0)\). We first determine this number.
Consider \(\mathcal P=\{(z,t)\in A\times H:z\in B_{a_t}\}\).
We count \(|\mathcal P|\) in two ways. For each \(z\in A\), exactly four
selected left blocks contain \(z\), and by
Lemma~\ref{lem:three-out-of-four}, exactly three of the corresponding
right blocks contain \(\mathbf0\). Since \(|A|=6\),
\(
|\mathcal P|=6\cdot3=18.
\)
On the other hand, \(A\) is unused, so \(A\neq B_{a_t}\) for every
\(t\in H\). By Lemma~\ref{lem:block-system}\textup{(iv)},
\(|A\cap B_{a_t}|=2\). Hence
\(|\mathcal P|=\sum_{t\in H}|A\cap B_{a_t}|=2|H|\).
Therefore \(18=2|H|\), so \(d_R(\mathbf0)=|H|=9\).

We now show that this is impossible. By
Lemma~\ref{lem:block-system}\textup{(v)}, exactly six block types contain
\(\mathbf0\). Since nine selected right-block occurrences contain
\(\mathbf0\), one of these six block types, say \(C\), occurs at most
once among \(B_{b_1},\ldots,B_{b_{12}}\).
Count incidences between the six points of \(C\) and the twelve selected
right blocks. At most one selected right block is equal to \(C\), and
hence contributes six incidences. Every other selected right block is
distinct from \(C\) and contributes exactly two. Therefore
\[
\sum_{y\in C}d_R(y)
\le
6+11\cdot2
=
28.
\]

On the other hand, \(\mathbf0\in C\) and
\(d_R(\mathbf0)=9\). Each of the other five points of \(C\) has right
degree at least four by Lemma~\ref{lem:tight-unused-block}\textup{(i)}.
Hence
\[
\sum_{y\in C}d_R(y)
\ge
9+5\cdot4
=
29.
\]
Thus the same quantity is at most \(28\) and at least \(29\), a
contradiction. Therefore no representation with at most twelve terms
exists, and
\(
\pf(K_{3,3}\sqcup K_{3,3})\ge13.
\)
\end{proof}

\noindent\textit{Proof of Theorem~\ref{thm:main}.}
Theorem~\ref{thm:main} follows from
Theorems~\ref{thm:upper-bound} and~\ref{thm:lower-bound}.
\hfill\(\square\)

\section{Concluding remarks}

The matrix formulation extends to disjoint unions of bipartite graphs. The lower bound proved here uses the particular structure of the coefficient vectors of \(K_{3,3}\).

Our result also shows that the Pfaffian number is not always multiplicative under disjoint union, since \(\pf(K_{3,3}\sqcup K_{3,3})=13<16=\pf(K_{3,3})^2\). 
It would be interesting to find the exact Pfaffian number of three copies of \(K_{3,3}\), and more generally of any number of copies.

The possible values of the Pfaffian number are also not well understood. Norine showed that \(2,3,\) and \(5\) do not occur, while \(4\) and \(6\) are known to occur~\cite{norine2009drawing,miranda2011matching}. By Theorem~\ref{thm:main}, \(13\) also occurs. It remains open whether \(7,9,\) or \(11\) occur, and whether infinitely many odd Pfaffian numbers occur.

\appendix
\section{Omitted proofs}
\label{app:proofs}

\subsection{Coefficient vectors of \(K_{3,3}\)}
\label{app:coefficient-vectors}

\begin{proof}[Lemma~\ref{lem:one-copy-coefficient-vectors}]
Let \(S=(s_{ij})\) be a sign matrix. Multiplying the six coordinates of
\(v_S\) gives
$$
\prod_{\pi\in S_3}v_S(\pi)
=
\left(\prod_{\pi\in S_3}\sgn(\pi)\right)
\left(\prod_{\pi\in S_3}\prod_{i=1}^{3}s_{i,\pi(i)}\right)
=
-\,\prod_{i,j=1}^{3}s_{ij}^{\,2}
=
-1,
$$
since \(S_3\) has three even and three odd permutations, and each
\(s_{ij}\) occurs in exactly two permutation products.
Conversely, let \(v=(v_1,\ldots,v_6)\in\{\pm1\}^6\) satisfy
\(\prod_{j=1}^{6}v_j=-1\). Consider

$$
S=
\begin{pmatrix}
p&q&r\\
u&t&1\\
1&1&1
\end{pmatrix}.
$$
In the permutation order \eqref{eq:permutation-order}, the coefficient
vector of \(\det(S\circ B)\) is
$$
(pt,\,-p,\,-qu,\,q,\,ru,\,-rt).
$$
Set
$$
p=-v_2,\qquad q=v_4,\qquad
t=-v_1v_2,\qquad u=-v_3v_4,\qquad
r=-v_3v_4v_5.
$$
Then the first five coordinates are \(v_1,\ldots,v_5\), while

$$
-rt=-v_1v_2v_3v_4v_5=v_6,
$$
where the last equality follows from
\(v_1v_2v_3v_4v_5v_6=-1\). Hence \(v_S=v\).
\end{proof}

\begin{proof}[Corollary~\ref{cor:normalized-coefficient-vectors}]
Let \(\varepsilon\) be the first coordinate of \(v_S\). Then
\(\varepsilon v_S\) has first coordinate \(1\), so there are unique
\(a_1,\ldots,a_4\in\F_2\) for which its next four coordinates are
\((-1)^{a_1},\ldots,(-1)^{a_4}\). Since the product of its six
coordinates is \(-1\), its last coordinate is
\(-(-1)^{a_1+a_2+a_3+a_4}\). Thus, for
\(a=(a_1,a_2,a_3,a_4)\),
$$
\varepsilon v_S
=
\begin{pmatrix}
1&
(-1)^{a_1}&
(-1)^{a_2}&
(-1)^{a_3}&
(-1)^{a_4}&
-(-1)^{a_1+a_2+a_3+a_4}
\end{pmatrix}^{\mathsf T}
=o_a.
$$
Hence \(v_S=\varepsilon o_a\), and both \(a\) and \(\varepsilon\) are
unique.
\end{proof}

\subsection{Proof of the block-system lemma}
\label{app:block-system}

\begin{proof}[Lemma~\ref{lem:block-system}]
Fix \(a,x\in \F_2^4\), and let \(d=|x+a|\). Since
\(\sum_{i=1}^{4}(-1)^{x_i+a_i}=4-2d\) and
\(|x|+|a|\equiv d\pmod 2\),
$$
e_x^{\mathsf T}o_a
=
1+\sum_{i=1}^{4}(-1)^{x_i+a_i}-(-1)^{|x|+|a|}
=
5-2d-(-1)^d.
$$
For \(d=0,1,2,3,4\), the corresponding values are
\(4,4,0,0,-4\). Hence \(f_a(x)\neq0\) exactly when
\(|x+a|\in\{0,1,4\}\), which is equivalent to
\(x+a\in D\), or \(x\in B_a\). Thus
$$
\operatorname{supp}(f_a)=B_a.
$$
If \(x\in B_a\), then \(|e_x^{\mathsf T}o_a|=4\). Together with
\eqref{eq:even-even-inner-product} and
\eqref{eq:normalized-function}, this gives
$$
|f_a(x)|
=
\begin{cases}
\dfrac13,&x=\mathbf0,\\
1,&x\neq\mathbf0.
\end{cases}
$$
This proves parts \textup{(i)} and \textup{(ii)}. Part \textup{(iii)}
follows immediately from \(B_a=a+D\) and \(|D|=6\).
For part \textup{(iv)}, let \(a\neq b\) and put
\(z=a+b\neq\mathbf0\). Translation gives
$$
|B_a\cap B_b|=|D\cap(z+D)|.
$$
The points in this intersection correspond to ordered representations
\(z=d+d'\) with \(d,d'\in D\). Every nonzero \(z\in \F_2^4\) has exactly
two such representations:
$$
\begin{array}{c|c}
|z| & \text{ordered representations}\\
\hline
1 & \mathbf0+\varepsilon_i,\ \varepsilon_i+\mathbf0\\
2 & \varepsilon_i+\varepsilon_j,\ \varepsilon_j+\varepsilon_i\\
3 & \mathbf1+\varepsilon_i,\ \varepsilon_i+\mathbf1\\
4 & \mathbf0+\mathbf1,\ \mathbf1+\mathbf0
\end{array}
$$
where in the third row \(i\) is the unique zero coordinate of \(z\).
Therefore \(|B_a\cap B_b|=2\).
Finally, \(x\in B_a\) exactly when \(a=x+d\) for some \(d\in D\).
Since \(|D|=6\), exactly six blocks contain \(x\), proving
part \textup{(v)}.
\end{proof}

\subsection{The connected example}
\label{app:connected-upper}

\begin{proof}[Upper bound in Theorem~\ref{thm:connected-matching-covered}]
By Theorem~\ref{thm:upper-bound},
\[
\pf(K_{3,3}\sqcup K_{3,3})\le 13.
\]
Let \(B_1\) and \(B_2\) be the weighted biadjacency matrices of the two
copies of \(K_{3,3}\). The weighted biadjacency matrix of their disjoint
union is the block-diagonal matrix
\(
B_1\oplus B_2.
\)
Hence, by Lemma~\ref{lem:bipartite-matrix-formulation}, there exist
coefficients \(c_1,\ldots,c_{13}\) and sign matrices
\(S_1,\ldots,S_{13}\in\{\pm1\}^{6\times6}\) such that
\[
\operatorname{per}(B_1\oplus B_2)
=
\sum_{s=1}^{13}
c_s\det\bigl(S_s\circ(B_1\oplus B_2)\bigr).
\]
Since \(B_1\oplus B_2\) is block diagonal, only the two diagonal
\(3\times3\) blocks of \(S_s\) are relevant. Denote these blocks by
\(S_{1,s}\) and \(S_{2,s}\). Then
\[
\operatorname{per}(B_1)\operatorname{per}(B_2)
=
\sum_{s=1}^{13}
c_s\,
\det(S_{1,s}\circ B_1)
\det(S_{2,s}\circ B_2).
\]

We show that, for each \(s\in[13]\), the two signings
\(S_{1,s}\) and \(S_{2,s}\) can be extended to a signing of
\(\widetilde G\) without changing the sign of any corresponding
perfect matching. The same coefficients \(c_1,\ldots,c_{13}\) will
then give a thirteen-term Pfaffian representation of
\(\widetilde G\).

Fix \(s\in[13]\).
For an edge \(e\) of \(\widetilde G\), write
\(\operatorname{sgn}_s(e)\in\{\pm1\}\) for the sign assigned to \(e\)
in this extension.
In the \(i\)-th copy, regard
\(a_i,b_i,c_i,y_i\) as row vertices and
\(\alpha_i,\beta_i,\gamma_i,x_i\) as column vertices.
Keep the signs of all original edges other than
\(a_i\alpha_i\) unchanged. On the path
\[
a_i-x_i-y_i-\alpha_i,
\]
assign sign \(+1\) to \(x_iy_i\), and choose the signs of
\(a_ix_i\) and \(y_i\alpha_i\) so that
\[
\operatorname{sgn}_s(a_ix_i)\,
\operatorname{sgn}_s(y_i\alpha_i)
=
-\,S_{i,s}(a_i,\alpha_i).
\]
Finally, choose the signs of the two crossed edges so that
\[
\operatorname{sgn}_s(x_1y_2)\,
\operatorname{sgn}_s(x_2y_1)
=
-1.
\]
We claim that every perfect matching of \(\widetilde G\) has, in this
extended signing, the same sign as the corresponding perfect matching
of \(K_{3,3}\sqcup K_{3,3}\).

First observe that every perfect matching of \(\widetilde G\) uses
either both crossed edges or neither of them. Indeed, using exactly one
crossed edge would leave unequal numbers of unmatched row and column
vertices in each of the two copies.

Suppose first that neither crossed edge is used. In the \(i\)-th copy,
the matching either uses \(x_iy_i\), or it uses both
\(a_ix_i\) and \(y_i\alpha_i\).
If it uses \(x_iy_i\), deleting this edge gives a perfect matching
\(M_i\) of \(K_{3,3}\) avoiding \(a_i\alpha_i\). Since the new row
\(y_i\) is matched to the new column \(x_i\), the permutation sign is
unchanged, and \(x_iy_i\) has sign \(+1\). Thus the sign of the
matching is unchanged.
Now suppose that it uses \(a_ix_i\) and \(y_i\alpha_i\). Replacing
these two edges by \(a_i\alpha_i\) gives a perfect matching \(M_i\)
of \(K_{3,3}\). Relative to the matching obtained by using
\(a_i\alpha_i\) together with \(y_ix_i\), the assignments of the two
columns \(\alpha_i\) and \(x_i\) are exchanged. Hence the permutation
sign changes by a factor of \(-1\). The ratio of the edge-sign
contributions is
\[
\frac{
\operatorname{sgn}_s(a_ix_i)
\operatorname{sgn}_s(y_i\alpha_i)}
{S_{i,s}(a_i,\alpha_i)}
=
-1.
\]
Therefore the total change in sign is
\(
(-1)(-1)=1.
\)
Thus the sign is unchanged in this case as well.

It remains to consider a perfect matching using both crossed edges.
After deleting the crossed edges, the remaining edges correspond to
perfect matchings \(M_1,M_2\) of the two copies of \(K_{3,3}\), both
avoiding the distinguished edges \(a_i\alpha_i\).
Compare such a matching with the matching obtained from
\(M_1,M_2\) by using \(x_1y_1\) and \(x_2y_2\) instead of the crossed
edges. In the latter matching we have the two assignments
\[
y_1\mapsto x_1,
\qquad
y_2\mapsto x_2,
\]
whereas the crossed matching has
\[
y_1\mapsto x_2,
\qquad
y_2\mapsto x_1.
\]
Thus the two new column assignments are exchanged, so the permutation
sign changes by a factor of \(-1\). On the other hand, since
\(x_1y_1\) and \(x_2y_2\) both have sign \(+1\), the ratio of the
edge-sign contributions is
\[
\frac{
\operatorname{sgn}_s(x_1y_2)
\operatorname{sgn}_s(x_2y_1)}
{
\operatorname{sgn}_s(x_1y_1)
\operatorname{sgn}_s(x_2y_2)}
=
-1.
\]
Again the total change in sign is
\(
(-1)(-1)=1.
\)

Therefore, for each \(s\in[13]\), every perfect matching of
\(\widetilde G\) has the same coefficient in the extended signed
determinant as the corresponding perfect matching of
\(K_{3,3}\sqcup K_{3,3}\) has in the original \(s\)-th term.
Since the original thirteen terms, with coefficients
\(c_1,\ldots,c_{13}\), give coefficient \(1\) to every perfect
matching of \(K_{3,3}\sqcup K_{3,3}\), the same thirteen coefficients
give coefficient \(1\) to every perfect matching of
\(\widetilde G\).
Hence the thirteen extended signings form a Pfaffian representation of
\(\widetilde G\), and therefore
\[
\pf(\widetilde G)\le13.
\]
\end{proof}

\begin{credits}
\subsubsection{\ackname}
We thank Alberto Miranda and Cl{\'a}udio Lucchesi for their helpful comments and suggestions. Priyanshu Pant is supported by the CSIR--UGC NET Junior Research Fellowship. This work is also supported by ANRF Grant ANRF/\allowbreak ARGM/\allowbreak 2025/\allowbreak 002341/\allowbreak TS.
\end{credits}

\bibliographystyle{splncs04}
\bibliography{references}

@article{norine2009drawing,
 author={Norine, Serguei},
 title={Drawing 4-{Pfaffian} Graphs on the Torus},
 journal={Combinatorica}, volume={29}, number={1}, pages={109--119}, year={2009}, doi={10.1007/s00493-009-2354-0}
}

@article{miranda2011matching,
 author={Miranda, Alberto Alexandre Assis and Lucchesi, Cl{\'a}udio Leonardo},
 title={Matching Signatures and {Pfaffian} Graphs},
 journal={Discrete Mathematics}, volume={311}, number={4}, pages={289--294}, year={2011}, doi={10.1016/j.disc.2010.10.014}
}

@article{robertson1999permanents,
 author={Robertson, Neil and Seymour, Paul D. and Thomas, Robin},
 title={Permanents, {Pfaffian} Orientations, and Even Directed Circuits},
 journal={Annals of Mathematics}, volume={150}, number={3}, pages={929--975}, year={1999}, doi={10.2307/121059}
}

@article{mccuaig2004polya,
 author={McCuaig, William}, title={P{\'o}lya's Permanent Problem},
 journal={The Electronic Journal of Combinatorics}, volume={11}, number={1}, pages={R79}, year={2004}, doi={10.37236/1832}
}

@article{valiant1979complexity,
 author={Valiant, Leslie G.}, title={The Complexity of Computing the Permanent},
 journal={Theoretical Computer Science}, volume={8}, number={2}, pages={189--201}, year={1979}, doi={10.1016/0304-3975(79)90044-6}
}

@article{cayley1849determinants,
 author={Cayley, Arthur}, title={Sur les d{\'e}terminants gauches},
 journal={Journal f{\"u}r die reine und angewandte Mathematik}, volume={38}, pages={93--96}, year={1849}, doi={10.1515/crll.1849.38.93}
}

@incollection{kasteleyn1967graph,
 author={Kasteleyn, P. W.}, title={Graph Theory and Crystal Physics},
 booktitle={Graph Theory and Theoretical Physics}, editor={Harary, Frank}, publisher={Academic Press}, address={London}, pages={43--110}, year={1967}
}

@article{little1975characterization,
 author={Little, C. H. C.}, title={A Characterization of Convertible \((0,1)\)-Matrices},
 journal={Journal of Combinatorial Theory, Series B}, volume={18}, number={3}, pages={187--208}, year={1975}, doi={10.1016/0095-8956(75)90048-9}
}

@misc{pant2026exponential,
 author={Pant, Priyanshu and Singh, Ranveer}, title={Exponential Lower Bounds for the {Pfaffian} Number of Graphs},
 year={2026}, eprint={2605.21077}, archivePrefix={arXiv}, primaryClass={math.CO}, note={To appear in ITCS 2027, arXiv:2605.21077}, doi={10.48550/arXiv.2605.21077}
}

@inproceedings{moco2021signature,
 author={Costa Mo{\c{c}}o, Roberta R. M. and Miranda, Alberto Alexandre Assis and Nunes da Silva, C{\^a}ndida},
 title={The Signature Matrix for 6-{Pfaffian} Graphs},
 booktitle={Proceedings of the XI Latin American Algorithms, Graphs and Optimization Symposium (LAGOS 2021)},
 series={Procedia Computer Science}, volume={195}, pages={298--305}, publisher={Elsevier}, year={2021}, doi={10.1016/j.procs.2021.11.037}
}

@inproceedings{junchaya2026lower,
 author={Junchaya, Enrique and Miranda, Alberto Alexandre Assis and Lucchesi, Cl{\'a}udio Leonardo},
 title={Lower Bounds for the {Pfaffian} Number of Graphs},
 booktitle={52nd International Workshop on Graph-Theoretic Concepts in Computer Science (WG 2026)},
 editor={Goedgebeur, Jan and Rz{\k{a}}{\.z}ewski, Pawe{\l}},
 series={Leibniz International Proceedings in Informatics (LIPIcs)}, volume={376}, pages={28:1--28:15},
 publisher={Schloss Dagstuhl -- Leibniz-Zentrum f{\"u}r Informatik}, address={Dagstuhl, Germany}, year={2026}, doi={10.4230/LIPIcs.WG.2026.28}
}

@book{lovasz1986matching,
 author={Lov{\'a}sz, L{\'a}szl{\'o} and Plummer, Michael D.}, title={Matching Theory},
 series={North-Holland Mathematics Studies}, volume={121}, publisher={North-Holland}, address={Amsterdam}, year={1986}
}

@book{lucchesi2024perfect,
 author={Lucchesi, Cl{\'a}udio L. and Murty, U. S. R.}, title={Perfect Matchings: A Theory of Matching Covered Graphs},
 series={Algorithms and Computation in Mathematics}, publisher={Springer}, address={Cham}, year={2024}, doi={10.1007/978-3-031-47504-7}
}

@misc{PantSinghPfaffianRecognition,
  author = {Pant, Priyanshu and Singh, Ranveer},
  title = {{Pfaffian} Recognition is in {coRP}},
  year = {2026},
  note = {arXiv:2610.29817},
  doi = {10.48550/arXiv.2610.29817}
}

\end{document}